\documentclass[12pt, reqno]{article}
\usepackage{amsmath, amsthm, a4, latexsym, amssymb}
\usepackage{geometry}
\usepackage{amsmath, amsfonts, amssymb, amsthm, scalerel}

\usepackage{bm}
\usepackage{setspace}
\usepackage{enumitem,multicol, multirow}
\usepackage{adjustbox }
\usepackage{booktabs, makecell,array} 
\usepackage{undertilde}
\usepackage{longtable}
\usepackage{lastpage}
\usepackage[english]{babel}
\newtheorem{theorem}{Theorem}[section]
\newtheorem{lemma}[theorem]{Lemma}

\usepackage{verbatim, float}
\DeclareUnicodeCharacter{0301}{\hspace{-1ex}\'{ }}
\DeclareUnicodeCharacter{0300}{\hspace{-1ex}\'{ }}
\usepackage[utf8]{inputenc}
\usepackage{fancyhdr}
\usepackage{array, bigints}
\usepackage[natbibapa]{apacite}
\usepackage{textcomp}
\usepackage{hyperref}
\usepackage{graphicx}
\usepackage{subcaption}
\theoremstyle{definition}

\newtheorem{rem}[theorem]{Remark}

\AtBeginDocument{\AtBeginShipoutNext{\AtBeginShipoutDiscard}}
\addtocounter{page}{-1}

\begin{document}
	
\begin{flushleft}	
\title{{DISTRIBUTION OF  THE ELEMENTAL  REGRESSION WEIGHTS  WITH t-DISTRIBUTED  CO-VARIATE MEASUREMENT ERRORS }}	

\author{I. Seidu$^1$,   E.  Nyarko$^1$,   S.  Iddi$^1$,  E. Ranganai$^2$,  and  K. Doku-Amponsah$^{1,3}$}

\end{flushleft}

\maketitle
\thispagestyle{empty}
\vspace{-0.5cm}

\renewcommand{\thefootnote}{}
\footnote{\textit{$^1$ University  of  Ghana, School  of  Physical  and Mathematical Sciences, Department  of  Statistics and  Actuarial  Science }}
\footnote{\textit{$^2$ 
		Department of Statistics, University of South Africa, Roodeport, South Africa}.}
\footnote{\textit{$^3$ Email address of  the corresponding author: kdoku-amponsah@ug.edu.gh}}  

\renewcommand{\thefootnote}{1}

\begin{abstract}
In this article, a heuristic approach is used to determined the best approximate distribution of  $\dfrac{Y_1}{Y_1 + Y_2}$, given that $Y_1,Y_2$  are  independent,  and   each of $Y_1$ and $Y$  is distributed as the $\mathcal{F}$-distribution with common denominator degrees of freedom. The proposed approximate distribution is subject to graphical comparisons and distributional tests. The  proposed  distribution   is  used to  derive  the  distribution of  the elemental regression  weight $\omega_E$,  where  $E$  is  the  elemental  regression  set.  
\end{abstract}


\vspace{0.3cm}

\pagebreak
\section{INTRODUCTION}\label{sec1}
The heuristic approach is mainly about using  analytical means, experience,  and  to some lesser extent "eyeballing" to provide an a quicker and sometimes a simpler solution to an otherwise complex problem \citep{abbass2001, estrada2008}. These methods are mostly used in computer science \citep{woods1996program, masehian2007classic, lee2008heuristic,  mohanty2021novel}, in data mining and machine learning \citep{hua2007cleaning, majumdar2016heuristic} and other areas such as management science and product designing \citep{kohli1987heuristic}. In probability theory and mathematical statistics the heuristic approach was used to provide an alternated proof to the Kolmogorov-Smirnov statistic \citep{doob1949heuristic}. Other works in literature that used the approach include \cite{indlekofer2002number} using it to find some asymptotic distributions in probabilistic number theory and \cite{yang2021heuristic} used the heuristic approach to solve problems with sampling and maintaining probability distribution of populations observations were taken from.

Elemental sets (ES) regression  involves regression which  is  subsets of the design matrix with dimensions equal to number of independent variables in a regression model. \cite{Hawkins1984} related the concepts Elemental set regression to least squares regression, quantile regression etc. \cite{Ranganai2007, Ranganai2014} extended the ideas of elemental sets regression to conducts post model diagnostics under quantile regression framework with non-stochastic predictors. The concepts of ES was further used in \cite{,Ranganai2017} for determining high leverage values in quantile regression when there is the presence of error-in-variables and it was shown that the concept is related to the Wilk's Lambda statistic.
\par
Distribution of functions or combinations of random variables has been studied thoroughly through probability theories books and literature. The common functions or combination include the sum, difference (in case of only two variables) and the product. Other combination is the variable as a proportion of two variables. In this case given the variables $X_1$ and $X_2$, the proportion of the two variates can be of the form $\dfrac{X_1}{X_1 + X_2}$ or $\dfrac{X_2}{X_1 + X_2}$. In most cases, is it the Beta distribution that better describes these proportion variables \cite{ casella2001statistical, Gupta2004} but the resulting shape parameters are far fetched. A common case is where $X_1 \sim \Gamma\left(\alpha_1, \zeta\right)$ and $X_2 \sim \Gamma\left(\alpha_2, \zeta\right)$, it is known that, the proportion variable $\dfrac{X_1}{X_1 + X_2}$ is distributed as the Beta with shape parameters $\alpha_1$ and $\alpha_2$ \citep{Gupta2004, Ross2020}.

On letting  $Y_1, Y_2, \ldots, Y_m$ be univariate Fisher-Snedecor random variables, some handful of literature have attempted to determine and approximate combinations of $Y_j$s. Specifically, the sum \citep{Dyer1982, Du2020, Yousuf2018, Gupta2004}, product \citep{Badarneh2020}.  The exact or the approximate combination $\dfrac{Y_j}{Y_j + Y_i}$ in terms of the degrees of freedoms is yet to be studied to the best known knowledge.  
Determining the distribution of combinations of  F-distribution  variates almost always results in complex number expressions, beyond intermediate mathematical expressions, and resulting expressions which are not applicable to statistically measurable variables using conventional means. As such, most of the combination of the $F-$ distribution such as the sum are approximated.
\par
In  this  article  we use the heuristic  approach  and  simulation  methods to determine an approximate distribution of  $\dfrac{Y_j}{Y_j + Y_i},$ for   $i\not=j,$  in terms of the degrees of freedoms. 

\subsection{Background~and~Motivation}\label{subsec1}
	Given the linear regression model, 
	\begin{equation}
		\bm{Y} = \bm{X}\mathcal{\bm{B}} + \bm{\varepsilon},
	\end{equation}
	
	where \(\bm{Y}\) is the response matrix,  \(\bm{X}\) is the design matrix, \(\mathcal{\bm{B}}\) is the regression parameter matrix, \(\varepsilon\) is the the random error matrix. Also  \(\bm{Y}\),  \(\bm{X}\), \(\mathcal{\bm{B}}\), \(\varepsilon\) are \(\ell\times 1\), \(\ell\times (\rho+1)\), \((\rho+1)\times 1\) and \(\ell\times 1\) matrices respectively, and $\bm{\varepsilon}$ has the classical assumption being $E\left(\bm{\varepsilon}\right)=\bm{0}_n$, $Var\left(\bm{\varepsilon}\right)= \bm{I}_n\sigma^2$. The most common method for estimating the \(\mathcal{\bm{B}}\) is ordinary least squares (OLS) method, where the \(\bm{\varepsilon} ^T \bm{\varepsilon}\) is minimized with respect to \(\mathcal{\bm{B}}\). 
	
	Before OLS was in use, the concept of elemental set regression was used by Boscovich (1757)  \citep{Koenker2005, Davino2014, Ranganai2007}who used the concept of to estimate distance of the median arc around Rome. Boscovich's estimation happened about half a century before the introduction of the popular OLS. On partitioning the \(\bm{Y}\) into $\displaystyle \bm{Y} = \begin{pmatrix}
		\bm{Y}_E \\
		\bm{Y}_R
	\end{pmatrix},$
	where \(\bm{Y}_E\) and \(\bm{Y}_R\) are \((\rho+1)\times 1\) and \((\ell-\rho-1)\times 1\)
	matrices respectively and partitioning $\bm{X}$into 
	$\bm{X} = \begin{pmatrix}
		\bm{X}_E \\
		\bm{X}_R
	\end{pmatrix} ,$	where \(\bm{X}_E\) and \(\bm{X}_R\) are \((\rho+1)\times (\rho+1)\) and \((\ell-\rho-1)\times (\rho+1)\) matrices respectively,  the set \((\bm{X}_E, \bm{Y}_E)\) is the $E^{th}$ elemental set (ES) and its corresponding regression model, known as the ES regression 	
	For any linear regression model with $\rho$ predictors and $\ell$ observations, there are exactly $\displaystyle L = \begin{pmatrix}
		\ell\\
		\rho+1
	\end{pmatrix} $ elemental set regressions. For each elemental set, the elemental regression weight (ERW) for set \(E\) denoted by $\omega_E$ is defined as 
	\begin{equation}
		\omega_E = \dfrac{\left|\bm{X}_E^T \bm{X}_E \right|}{\left|\bm{X}^T\bm{X} \right|} = \dfrac{\left|\bm{X}_E^T \bm{X}_E \right|}{\left|\bm{X}_E^T\bm{X}_E  + \bm{X}_R^T\bm{X}_R\right|}
	\end{equation}

	To detect multiple high leverage values associated with regression quantiles in order to prevent swamping and masking as discussed by \cite{Hadi1992, Rocke1996, Rousseeuw1990, Roberts2015} , \cite{Ranganai2007} proposed a statistics based on the elemental regression weight, $\omega_E$.	Given that $ \bm{X}  \sim t_{\rho}(\bm{0}, \bm{\Sigma}, \mathit{\nu}) $
	is a $\rho$-variate t-distribution with $n_1+n_2$ observations.
	The distribution of 
	$ \Omega = \dfrac{\left|\bm{X}_1^T\bm{X}_1\right|}{\left|\bm{X}^T\bm{X}\right|}, $ where $ \bm{X}_1  \sim t_{\rho}](\bm{0}, \bm{\Sigma}, \mathit{\nu}) $ a $\rho$-variate t-distribution with $n_1$ observations can be analogously determined by
	adopting a method by 	\cite{Mardia1979}. \\
	
	Let \(Q_1=\bm{X}_1^T\bm{X}_1\) and 
	\(Q_2=\bm{X}_2^T\bm{X}_2\), where 
	\(X_2 \)  is   also a $\rho$-variate t-distribution with $n_2$ observations. We  write 	$\displaystyle 	\Omega= |\bm{Q}_1|/|\bm{Q}_1+ \bm{Q}_2|$     and   define the  sequence $\{M_i\}_{i}^{n_2}$  by the  relation	$\bm{M}_i- \bm{M}_{i-1} =\bm{X}_i\bm{X}_i^T,\hspace*{0.2in} \bm{M}_0 =\bm{Q}_1 , ,\hspace*{0.2in} \bm{M}_{n_2} =\bm{Q}_2$, where $\bm{x}_i$  means the first i  component  of  $x$.	Then,  note  that we  have   
	\begin{equation}\label{eq:4}
		\Omega = \dfrac{|\bm{M}_0 |}{|\bm{M}_{n_2}|}=\dfrac{|\bm{M}_0|}{|\bm{M}_1|} . \dfrac{|\bm{M}_1|}{|\bm{M}_2|} . \dfrac{|\bm{M}_2|}{|\bm{M}_3|} . \ldots  
		\dfrac{|\bm{M}_{n_2-1}|}{|\bm{M}_{n_2}|} 
	\end{equation}
	
	We  let 	$ t_i= \dfrac{|\bm{M}_{i-1}|}{|\bm{M}_i|} $  and  write 	
	$	\Omega: = \prod\limits_{i=1}^{n_2}t_i.$    Now,	from Soch (2021) and \cite{Kotz2004}, \(\bm{ X}_i^T\bm{X}_i\) is distributed as the \(F\) matrix distribution   \(F{(\rho,\mathit{\nu})}\) and \(\bm{ X}_1^T\bm{ X}_1\) distributed as \(F\) matrix distribution   \(F{(n_1 -\rho,\mathit{\nu})}\) with dispersion matrix \(\bm{\Sigma}\).  	This implies that \(\bm{M}_{i-1} \sim F{(n_1 -\rho+i,\mathit{\nu})}\). Thus,  $t_i$  is  of  the  form

	\begin{equation}
		t_i = \dfrac{Y_{1(i)}}{Y_{1(i)} +
			Y_2},
	\end{equation}
	
	where \(Y_{1(i)} \sim F{(n_1 -\rho+ i,\mathit{\nu})}\) and \(Y_2 \sim F{(\rho,\mathit{\nu})}\). 
	Thus the first step in estimating the distribution of $\omega_{E}$ under multivariate $t$ measurement error-in-variable  is  to determine the distribution of $	t_i = \dfrac{Y_{1(i)}}{Y_{1(i)} +
		Y_2}$ which does not exist in literature, especially in terms of the degrees of freedoms of resulting F distributions.  The  first  results  of  this  article  is  the  distribution  of  the  $t_i$. 
		The  remaining  part  of  the  article  is  structured  in  this  way:  In Section~\ref{sec2}  we  present  the  main  results,  Theorem~\ref{main1}~  and  Theorem~\ref{main2}. Section~\ref{sec3}  contains  the  proof  of  the  main  Theorems. Graphical  Comparison  of  Theorem`\ref{main1}  is  given  in  Section~\ref{sec4}  In  the  last  section,  Section~\ref{sec5},  we  give a  discussion of  the  main  results and   conclusion. 		
		
	\pagebreak	
\section{MAIN RESULTS}\label{sec2}
We  begin  the section  by  introducing a new concept  of  \emph {independence} as  follows::

Suppose $X$  and  $Y$ are  jointly  distributed random  variables  with   probability density function (p.d.f), $h(x,y)$  and  let  $D_h$   the  support    of  $h$. We   shall  call  $X$  and  $Y$  

\begin{itemize}
	\item[(i)] {\bf Weakly Sub-Independent} iff  there     exists    $ \lambda_1,\lambda_2\in(0,\infty)$   and  p.d.f's  $f_1^X,f_2^X,f_1^Y$,  $f_2^Y,$   such  that  the  joint  probability density function   of  $X$  and  $Y$ satisfies 
	
	\begin{equation}\label{sub0}\lambda_1 f_1^X(x)f_1^Y(y)\le h(x,y)\le \lambda_2f_2^X(x)f_2^Y(y),\quad  \mbox{$\forall (x,y)\in D_h.$}
	\end{equation}
	
	\item[(ii)] {\bf Sub-Independent}  iff  there     exists    $ \lambda_1,\lambda_2\in(0,\infty)$   and  p.d.f's   $f_1^X,f_2^X,f_1^Y$, $f_2^Y$  such  that  the  joint  probability density function  of  $X$  and  $Y$ satisfies 
	
	\begin{equation}\label{sub1}
		\lambda_1 f_1^X(x)f_2^Y(y)\le h(x,y)\le \lambda_2f_2^X(x)f_2^Y(y),\quad  \mbox{$\forall (x,y)\in D_h.$}
	\end{equation}
	{\bf OR}

	\begin{equation}\label{sub2}
		\lambda_1 f_2^X(x)f_1^Y(y)\le h(x,y)\le \lambda_2f_2^X(x)f_2^Y(y),\quad  \mbox{$\forall (x,y)\in D_h.$}
	\end{equation}
		{\bf OR} \,  both  \eqref{sub1} and \eqref{sub2} holds for  $h(x,y).$
\end{itemize}

\begin{rem}
	
	Note, from  the  our  proposed   concert  of  independence  that,  \\
	(i) If   $f_1^X=f_2^X$, $f_1^Y=f_2^Y$  and  $\lambda_1=\lambda_2=1$,  then  $X$  and  $Y$  are  \emph{ independent}.
	
	(ii) Every  \emph{sub-independent}  random  variables  $X$  and  $Y$  are  \emph{weakly  sub-independent},  and  every independent random  variables  $X$  and  $Y$ are sub-independent.
	
	(iii)     $X$ and  $Y$  are  \emph{weakly  sub-independent, sub-independent}   iff  the  joint cumulative  distribution function (c.d.f)  of  $X$  and  $Y$  satisfies \ref{sub0},   \ref{sub1}  and/or  \,   \ref{sub2},  respectively.

	(iv)  \emph{Sub-independence}  is  weaker  than  Independence,  and    not  all  \emph{sub-independent} random  variables   are  independent.
\end{rem}

	For any two  $X$ and  $Z$  we  shall  write  
$U(X,Z)=X+Z \quad \mbox{and}\quad  W(X,Z)=\frac{X}{X+Z}.$

	\begin{theorem}\label{prop1}\label{main1}
		Suppose $Y_1\sim F(m_1,\mathfrak{\nu})$  and  $Y_2\sim F(m_1,\mathfrak{\nu}),$  while   $\displaystyle \mathfrak{\nu}>m_1\ge m_2. $    Let  $ h_{\quad\mathfrak{\nu}}^{m_1,m_2}(w)$   be  the   probability  density function  of   the random variable 
		$$ W:=W(Y_1,Y_2).$$  Then,   the  following  statements  hold about  the  function  $$\displaystyle \Lambda(w):=\lim_{ m_1\to (m_2+1/2)^{+}}\Big[\int_{0}^{w} h_{\quad\mathfrak{\nu}}^{m_1,m_2}(t) dt\Big],\,\,  0<w<1:$$
		
		\begin{itemize}
			
			\item[(i)]	As  $m_1\to (m_2+1/2)^+$   we  have  
			$$\Lambda(w) \approxeq\int_{0}^{w} \dfrac{t^{\frac{m_2+1/2}{2}-1}(1-t)^{\frac{m_2}{2}-1}}{ B\left(\dfrac{m_2+1/2}{2}, \dfrac{m_2}{2}\right)}dt,\quad \mbox{$0<w<1.$}$$
			
			\item[(ii)]    $\Lambda(w)$  is  absolutely  continuous.  
		\end{itemize}
		
	\end{theorem}
	We  recall  from  Subsection~\ref{subsec1}  that  $\displaystyle \Omega=\prod_{i=1}^{n_2}t_i,$  where  $t_i=\frac{Y_{1(i)}}{Y_{1(i)}+Y_2}$,   \(Y_{1(i)} \sim F{(n_1 -\rho+ i,\mathit{\nu})}\) and \(Y_2 \sim F{(\rho,\mathit{\nu})}\).  From  Theorem~\ref{main1},   $t_i$  is  approximately  distributed  as  
	  $$B_{t_i}\left(\dfrac{1}{2}(\rho+0.5), \dfrac{1}{2}\rho\right): = \dfrac{\left(t_i\right) ^{(\rho-0.5)/2} \left(1-t_i\right) ^{(\rho-2)/2} }{B\left(\dfrac{\rho+0.5}{2}, \dfrac{\rho}{2}\right)},\mbox{$0<t_i <1$}.$$  Recall  from \cite[pp.~722]{Springer1970} (or  the  references therein)  the     Meijer G- function  is  given  by the  contour  integrals: 
	  
	  $\begin{aligned} &G_{p,q}^{\alpha,\beta} \left(\omega \left| \begin{array}{llll}
	  	a_1,
	  	& a_2, & \ldots, & a_{\alpha} \\
	  	b_1,& b_2, & \ldots, & b_{\beta}
	  \end{array}  \right. \right)=\frac{1}{2\pi i}\int_{c-i^{\infty}}^{c+i^{\infty}}\dfrac{\omega^s\prod_{i=1}^{p}(s+b_i)\prod_{i=1}^{q}(1-a_i-s)}{\prod_{i=q+1}^{\alpha}(s+a_i)\prod_{i=p+1}^{\beta}(1-b_i-s)}ds,\end{aligned}$
	  
where  $c$  is a real  constant  defining  the   Bromvich   path separating the poles of $\Gamma(s+b_i)$  from  $Gamma(1-a_i-s)$  and  s) and where the empty product $\prod_ {j=r}^{r-1}$  defined to be unity. Theorem~\ref{cor1}   below gives  an  approximate  distribution  of   $\Omega$   based  on  the  approximate distribution  of  $W$.

\begin{theorem}\label{cor1}
	The  following  hold  for  the probability density function  and  the  cumulative  distribution  function   of  $\Omega$:  
		
(i) The p.d.f  of  $\Omega$  is  approximately given  by

$$	\begin{aligned}
	\sigma (\omega)&= \Big[
	\frac{\Gamma\left(\frac{1}{2}(2\rho+0.5)\right)}{\Gamma\left(\frac{1}{2}(\rho+0.5)\right)}\Big]^{n_2}\\
	&\qquad\qquad\qquad\times G^{n_2,0}_{n_2,0}
	\left(\omega \left| \begin{array}{llll}
		\frac{1}{2}(2\rho - 1.5 )
		& \frac{1}{2}(2\rho -1.5 ), & \ldots, & \frac{1}{2}(2\rho -1.5) \\
		\frac{1}{2}(\rho-1.5)& \frac{1}{2}(\rho-1.5), & \ldots, & \frac{1}{2}(\rho-1.5)
	\end{array}  \right. \right).
\end{aligned}$$

(ii)  The  c.d.f  of  $\Omega=\omega_E,$ where  $n_2=\ell-\rho$,  is  approximately given  by

	$$	\begin{aligned}
		&	\Sigma(\omega)=  \Big[
		\frac{\Gamma\left(\frac{1}{2}(2\rho+0.5)\right)}{\Gamma\left(\frac{1}{2}(\rho+0.5)\right)}\Big]^{\ell-\rho}\times\\
			&G^{\ell-\rho+1,1}_{\ell-p,1} \left(\omega \left| \begin{array}{llllll}
				1&\frac{1}{2}(2\rho + 0.5 )&\frac{1}{2}(\rho+0.5) , & \ldots, & \frac{1}{2}(2\rho + 0.5)&\frac{1}{2}(2\rho + 0.5) \\
				\frac{1}{2}(\rho+0.5)&	\frac{1}{2}(\rho+0.5)& \frac{1}{2}(\rho+0.5), & \ldots, & \frac{1}{2}(\rho+0.5)& 0
			\end{array}\right.\right).
 		\end{aligned}$$

\end{theorem} 	
	
	\section{PROOF  OF MAIN RESULTS}\label{sec3}
		\subsection{ Proof  of  Theorem~\ref{main1} using  Lemmas~\ref{L1},~\ref{L3} and ~\ref{L2}}
	Lemmas~\ref{L1},~\ref{L3}  and ~\ref{L2}  which   are  key  to  the  proof  of    Theorem~\ref{main1}  are  presented  with  proofs  in  this  subsection.  

\begin{lemma}\label{them1}\label{main2}\label{L1}
	Suppose $Y_1\sim F(m_1,\mathfrak{\nu}_1)$  and  $Y_2\sim F(m_1,\mathfrak{\nu}_2).$   Assume  $\displaystyle m_1-m_2> \displaystyle \mathfrak{\nu}_2-\mathfrak{\nu}_1 \quad    ,\dfrac{m_1}{\mathfrak{\nu}_1}\geq \dfrac{m_2}{\mathfrak{\nu}_2},\, \mbox{$m_1-m_2+2\mathfrak{\nu}_1>0,$ }$
	and   $\mathfrak{\nu}_2>m_1.$   Let  $h(u,w)$  be  the  joint  probability density function  of  $U:=U(Y_1,Y_2)$  and  $ W:=W(Y_1,Y_2) $.  Then,  $U$  and $W$  are   {\bf sub-independent}.
\end{lemma}
\begin{proof}
To  begin  we  write  	$\displaystyle K_{0}: = \dfrac{\left(\dfrac{m_1}{\mathfrak{\nu}_1}\right)^{\frac{m_1}{2}}\left(\dfrac{m_2}{\mathfrak{\nu}_2}\right)^{\frac{m_2}{2}}}{B\left(\dfrac{m_1}{2}, \dfrac{\mathfrak{\nu}_1}{2}\right) B\left(\dfrac{m_2}{2}, \dfrac{\mathfrak{\nu}_2}{2}\right)}$    and   we  write   	$$\displaystyle \phi\left(w\right) = \dfrac{w^{\frac{m_1}{2}-1}(1-w)^{\frac{m_2}{2}-1}}{ B\left(\dfrac{m_1}{2}, \dfrac{m_2}{2}\right)}, \hspace*{0.1in} 0<w<1.$$   We  show  that, the   joint  p.d.f  of $U$  and  $W$,  $h(u,w)$   satisfies 
	\begin{equation}\label{equ:2.4}	A_2 \pi_{2}(u)\phi (w)\leq h(u,w) \leq 	A_1 \pi_{1}(u)\phi (w),\end{equation}
	$$\pi _1\left(u\right) = \dfrac{\left(\dfrac{m_2}{\mathfrak{\nu}_2}\right)^{\frac{m_1+m_2}{2}}u^{\frac{m_1+m_2}{2}-1} \left(1+ \dfrac{m_2u}{\mathfrak{\nu}_2}  \right)^{-\frac{m_2+\mathfrak{\nu}_2}{2}}}{B\left(\dfrac{m_1+m_2}{2},\dfrac{ \mathfrak{\nu}_2 -m_1}{2}\right)}, \hspace*{0.1in} u>0$$
	,
	$$\pi _2\left(u\right) = \dfrac{\left(\dfrac{m_1}{2\mathfrak{\nu}_1}\right)^{\frac{m_1+m_2}{2}}u^{\frac{m_1+m_2}{2}-1} \left(1+ \dfrac{m_1u}{2\mathfrak{\nu}_1}  \right)^{-\left(m_1+\mathfrak{\nu}_1\right)}}{B\left(\dfrac{m_1+m_2}{2},\dfrac{ m_1-m_2 +2\mathfrak{\nu}_1 }{2}\right)}, \hspace*{0.1in} u>0,$$
	
	for  some real constants $A_1:=A_1(m_1,m_2)$  and   $A_2:=A_2(m_1,m_2).$  We  observe  that,
	
	the probability  density  functions    of \(Y_1,Y_2 \)   are  given  by
	\begin{equation}
		f_{Y_i}(y_1) = \dfrac{1}{B\left(\dfrac{m_1}{2}, \dfrac{\mathfrak{\nu}_1}{2}\right)}\left(\dfrac{m_1}{\mathfrak{\nu}_1}\right)^{\frac{m_1}{2}} y_1^{\frac{m_1}{2}-1}\left(1+\dfrac{m_1 y_1}{\mathfrak{\nu}_1}\right)^{-\frac{m_1+\mathfrak{\nu}_1}{2}} \label{eq:2.1}
	\end{equation}
	\begin{equation}
		f_{Y_i}(y_2) = \dfrac{1}{B\left(\dfrac{m_2}{2}, \dfrac{\mathfrak{\nu}_2}{2}\right)}\left(\dfrac{m_2}{\mathfrak{\nu}_2}\right)^{\frac{m_2}{2}} y_2^{\frac{m_2}{2}-1}\left(1+\dfrac{m_2 y_1}{\mathfrak{\nu}_2}\right)^{-\frac{m_2+\mathfrak{\nu}_2}{2}}, \label{eq:2.1a}
	\end{equation}respectively.  Note, 
	the Jacobian of the inverse  transformation from \((x_1,x_2))\) to \((u,w)\) ( \(x_1=uw\)  and \(x_2=u(1-w)\)) is given by  
	$\displaystyle  J(u,w) = \begin{vmatrix}
		w & u \\
		1-w & u
	\end{vmatrix} = u.$  We  use  the  Jacobean approach, to obtain 
	joint p.d.f.  of \(h(u,w)\)   as

	\begin{align}
		h(u,w)=K_{0} \scaleobj{1}{u}^{\frac{m_1 + m_2}{2}-1}\scaleobj{1}{w}^{\frac{m_1}{2}-1}&\scaleobj{1}{(1-w)}^{\frac{m_2}{2}-1}\nonumber\\ &\times\left[\scaleobj{1}{1+\dfrac{m_1uw}{\mathfrak{\nu}_1}}\right]^{-\frac{m_1+\mathfrak{\nu}_1}{2}}\left[\scaleobj{1}{1+\dfrac{m_2u(1-w)}{\mathfrak{\nu}_2}}\right]^{-\frac{m_2+\mathfrak{\nu}_2}{2}} \label{eq:2.4}	
	\end{align}

	Note that  exact  marginal probability  density function of $W$  is  given  by $\displaystyle \lambda (w) = \int_{0}^{\infty}h\left(u,w\right)du.$ The conventional integration methods cannot be used to determine the  $\lambda (w)$ explicitly. Thus,  $\lambda$ may exits only on  the  complex  plane.  The heuristic approach is thus employed to find an approximate form for $\lambda (w)$  as  follows: Consider the   joint pdf of \(h(u,w)\)  and replace the exponent $m_1 + \mathfrak{\nu}_1$  with $m_2 + \mathfrak{\nu}_2$ in Equation \eqref{eq:2.4}  to  obtain
	
	\begin{align}	
		h(u,w)\leq K_{0}&\scaleobj{1}{u}^{\frac{m_1 + m_2}{2}-1} \scaleobj{1}{w}^{\frac{m_1}{2}-1}\scaleobj{1}{(1-w)}^{\frac{m_2}{2}-1}\nonumber\\
		& \times\left\{\left[\scaleobj{1}{1+\frac{m_1uw}{\mathfrak{\nu}_1}} +\frac{m_2u(1-w)}{\mathfrak{\nu}_2} + \frac{m_1m_2u^2w(1-w)}{\mathfrak{\nu}_1\mathfrak{\nu}_2} \right]\right\}^{-\frac{m_2+\mathfrak{\nu}_2}{2}} \label{Thm1}
	\end{align}
	
	Note, the expression \(w(1-w)\), \(0\leq w\leq1\)   is  bounded  below   by  0 ,  hence  we  may   dominate the joint pdf \(h(u,w)\) above  by choosing the value of \(w\) that gives the least value for \(w(1-w)\). 
	
	Further replacing  \(\dfrac{m_1}{\mathfrak{\nu}_1}\) with \(\dfrac{m_2}{\mathfrak{\nu}_2}\), and $m_1$ with $m_2$ in Equation \eqref{Thm1}, we obtain
	
	\begin{equation}
		h(u,w)\leq K_0 B\left(\dfrac{m_1}{2}, \dfrac{m_2}{2}\right)  B_w\left(\dfrac{m_1}{2}, \dfrac{m_2}{2}\right) u^{\frac{m_1+m_2}{2}-1} \left(1+ \dfrac{m_2u}{\mathfrak{\nu}_2}  \right)^{-\frac{m_2+\mathfrak{\nu}_2}{2}} \label{eq:2.6}
	\end{equation}
	
	where  $\displaystyle B\left(a,b\right) = \int_{0}^{1} \eta ^{a-1}\left(1-\eta\right)^{b-1}d\eta = \dfrac{\Gamma(a)\Gamma(b)}{\Gamma(a+b)}$  and   $\displaystyle \Gamma (a)=\int_{0}^{\infty}s^{a-1}\exp(-s)ds$ is the gamma function.
	
	Multiply   and  dividing  right  hand side  of  \eqref{eq:2.6} by   $\displaystyle\int_{0}^{\infty}u^{\frac{m_1+m_2}{2}-1} \left(1+ \dfrac{m_2u}{\mathfrak{\nu}_2}  \right)^{-\frac{m_2+\mathfrak{\nu}_2}{2}}du $ and  noting that the Euler's Beta integral may be  given as $$B(\theta_1, \theta_2) = \int_{0}^{\infty} \xi ^ {\theta_1 -1} \left( 1+ \xi\right)^{-(\theta_1+\theta_2)}d\xi$$
	
	we  obtain

	\begin{equation}
		A_1 = \dfrac{\left(\dfrac{m_1\mathfrak{\nu}_2}{m_2\mathfrak{\nu}_1}\right)^{\frac{m_1}{2}}B\left(\dfrac{m_1+m_2}{2},\dfrac{ \mathfrak{\nu}_2 -m_1}{2}\right) B\left(\dfrac{m_1}{2}, \dfrac{m_2}{2}\right)}{B\left(\dfrac{m_1}{2}, \dfrac{\mathfrak{\nu}_1}{2}\right) B\left(\dfrac{m_2}{2}, \dfrac{\mathfrak{\nu}_2}{2}\right)} \label{eq:2.10}
	\end{equation}

	This gives the upper dominant joint PDF to be 
	\begin{equation}
		h(u,w)\leq  A_1\phi\left(u\right) \pi\left(w\right), \hspace*{0.1in} u>0, \hspace*{0.1in}, 0<w<1 \label{eq:2.7}
	\end{equation}
	
	where  $\pi_1$  and  $\phi$  are  defined  in   \eqref{equ:2.4}.		Using  similar  argument  as  those  of  the   upper bound   with  the exponent $m_2 + \mathfrak{\nu}_2$  by $m_1 + \mathfrak{\nu}_1$ in \eqref{eq:2.4}  we  obtain a  lower  bound  on  $h(u,v)$  as

	\begin{equation}
		h(u,w)\geq K_{0}\scaleobj{1}{u}^{\frac{m_1 + m_2}{2}-1} \scaleobj{1}{w}^{\frac{m_1}{2}-1}\scaleobj{1}{(1-w)}^{\frac{m_2}{2}-1} \left\{\left[\scaleobj{1}{1+\dfrac{m_1uw}{\mathfrak{\nu}_1}}\right]\left[\scaleobj{1}{1+\dfrac{m_2u(1-w)}{\mathfrak{\nu}_2}}\right]\right\}^{-\frac{m_1+\mathfrak{\nu}_1}{2}} \label{eq:2.11}   
	\end{equation}
	
	Now,  we  	dominate  \(h(u,w)\) further below by replacing,   \(\dfrac{m_2}{\mathfrak{\nu}_2}\) with \(\dfrac{m_1}{\mathfrak{\nu}_1}\) and take  \(w=\dfrac{1}{2}\) (the value  of  $w$  that  maximizes  $w(1-w)$  in \eqref{eq:2.11} to  obtain 
	, 
	
	\begin{equation}
		h(u,w)\geq K_{0} B\left(\dfrac{m_1}{2}, \dfrac{m_2}{2}\right)  B_w\left(\dfrac{m_1}{2}, \dfrac{m_2}{2}\right) u^{\frac{m_1+m_2}{2}-1} \left(1+ \dfrac{m_1u}{2\mathfrak{\nu}_1}  \right)^{-\left(m_1+\mathfrak{\nu}_1\right)} \label{eq:2.12}   
	\end{equation}
	
	Multiplying  and  dividing  right hand-side  of  \eqref{eq:2.11}  by  $$ \displaystyle \int _{0}^{\infty}u^{\frac{m_1+m_2}{2}-1} \left(1+ \dfrac{m_1u}{2\mathfrak{\nu}_1}  \right)^{-\left(m_1+\mathfrak{\nu}_1\right)}du$$ and Using the Euler's Beta integral  we  obtain

	\begin{equation}
		A_2 = \dfrac{2^{\frac{m_1 + m_2}{2}}\left(\dfrac{m_2\mathfrak{\nu}_1}{m_1\mathfrak{\nu}_2}\right)^{\frac{m_2}{2}}B\left(\dfrac{m_1+m_2}{2},\dfrac{ m_1-m_2 +2\mathfrak{\nu}_1 }{2}\right) B\left(\dfrac{m_1}{2}, \dfrac{m_2}{2}\right)}{B\left(\dfrac{m_1}{2}, \dfrac{\mathfrak{\nu}_1}{2}\right) B\left(\dfrac{m_2}{2}, \dfrac{\mathfrak{\nu}_2}{2}\right)} \label{eq:2.15}
	\end{equation}
	
	This gives the lower dominant  bound  of  $h(u,w)$  as
	\begin{equation}
		h(u,w)\geq  A_2\pi_2\left(u\right) \phi\left(w\right), \hspace*{0.1in} u>0, \hspace*{0.1in}, 0<w<1. \label{eq:2.13}
	\end{equation}

$\pi_2$	and $\phi$ are defined  above  in  \eqref{equ:2.4}.  Combining \eqref{eq:2.7} and   \eqref{eq:2.13} we  obtain  which  completes the proof.
\end{proof}

We  recall  the  definition  of  $A_1$  from \eqref{eq:2.10}  and  state   a  bound  on  the  marginal  probability density  function  (pdf)  of   $W$  as  follows:

\begin{lemma}\label{cor2}\label{L3}
	Suppose $Y_1\sim F(m_1,\mathfrak{\nu}_1)$  and  $Y_2\sim F(m_1,\mathfrak{\nu}_2).$   Assume  $\displaystyle m_1-m_2> \displaystyle \mathfrak{\nu}_2-\mathfrak{\nu}_1 \quad    ,\dfrac{m_1}{\mathfrak{\nu}_1}\geq \dfrac{m_2}{\mathfrak{\nu}_2},\, \mbox{$m_1-m_2+2\mathfrak{\nu}_1>0,$ }$
and   $\mathfrak{\nu}_2>m_1.$ Let  $h_{\quad\mathfrak{\nu}}^{m_1,m_2}(w)$  be  the   probability  density  function  of  $W$. Then,  $h_{\quad\mathfrak{\nu}}^{m_1,m_2}(w)$     satisfies 
	\begin{equation}
		\phi (w)\leq h_{\quad\mathfrak{\nu}}^{m_1,m_2}(w) \leq 	A_1\phi (w) \label{eq:2.18}	
	\end{equation}
\end{lemma}

\begin{proof}
	We  observe from  Lemma~\ref{L1} that,
	$	\pi_{2}(u)\phi (w)\leq h(u,w) \leq 	A_1 \pi_{1}(u)\phi (w) $, 
	
	$\displaystyle \phi\left(w\right) = \dfrac{w^{\frac{m_1}{2}-1}(1-w)^{\frac{m_2}{2}-1}}{ B\left(\dfrac{m_1}{2}, \dfrac{m_2}{2}\right)}, \hspace*{0.1in} 0<w<1.$
	Therefore,  integrating each term with respect to $u,$ and noting that 
	$\displaystyle \int_{0}^{\infty}\pi_{1}(u)du = \int_{0}^{\infty}\pi_{2}(u)du = 1,$   
	$\displaystyle\int_{0}^{\infty}h(u,w)du =  h_{\quad\mathfrak{\nu}}^{m_1,m_2}(w),$  we   obtain

	\begin{equation}
		\phi (w)\leq h_{\quad\mathfrak{\nu}}^{m_1,m_2}(w) \leq 	A_1\phi (w) \label{eq:2.18}	
	\end{equation}
	
	which  proves  Lemma~\ref{cor2}.

\end{proof}

Let  $\Delta :[0,1]\times[0,1]\to[0,1]$    be  the  total  variation  norm  on  the  space  of   distribution.   Thus, 

\begin{equation} \Delta[F_1 -F_2]:=\frac{1}{2}\int_{0}^{1}\Big |F_1(t)-F_2(t)\Big|dt\end{equation}\label{norm}
where  $F_1$  and  $F_2$  are  distribution  functions.

\begin{lemma}\label{them1}\label{L2}
	Suppose $Y_1\sim F(m_1,\mathfrak{\nu}_1)$  and  $Y_2\sim F(m_1,\mathfrak{\nu}_2).$   Assume  $\displaystyle m_1-m_2> \displaystyle \mathfrak{\nu}_2-\mathfrak{\nu}_1 \quad    ,\dfrac{m_1}{\mathfrak{\nu}_1}\geq \dfrac{m_2}{\mathfrak{\nu}_2},\, \mbox{$m_1-m_2+2\mathfrak{\nu}_1>0,$ }$
and   $\mathfrak{\nu}_2>m_1.$ Let  $\hat{\Lambda}_n $ and $\hat{\mathcal{K}}_n$  be  the empirical  distribution  functions  associated  with CDF's, $\Lambda$  and  $\mathcal{K}$  respectively. Then, as  $m_1\to (m_2+1/2)^+,$  we  have  
	
	\begin{equation}
		\Delta\Big[\hat{\Lambda}_n \left(w\right) - \hat{\mathcal{K}}_n(w)\Big] \leq \dfrac{A( m_2+1/2, m_2,  \mathfrak{\nu})}{n}.   \label{eq:L1}
	\end{equation}
\end{lemma}
\begin{proof}
	
	Note  from  Lemma~\ref{L3}  that   $$0 \leq\int_{0}^{w} h_{\quad\mathfrak{\nu}}^{m_1,m_2}(t) dt-\int_{0}^{w}\phi(t)dt\leq 	(A_1-1)\int_{0}^{w}\phi (t)dt .$$	
	Taking  limit  	as  $m_1\downarrow (m_2+1/2),$    we  have  
	\begin{equation}0 \leq\Lambda(w)-{\mathcal{K}}(w)\leq 	(A( m_2+1/2, m_2,  \mathfrak{\nu})-1){\mathcal{K}}(w).\label{eq:21}
	\end{equation}
	Applying  Glivenko-Cantelli's  Lemma   to  \eqref{eq:21}  we  obtain  
	$$	\Delta\Big[\hat{\Lambda}_n \left(w\right) - \hat{\mathcal{K}}_n \left(w\right)\Big]\leq \dfrac{A( m_2+1/2, m_2,  \mathfrak{\nu})}{n},$$ 
	
	which  ends  the  proof.
\end{proof}

 {\bf Proof  of  Theorem~\ref{main1}(i) via  Lemma~\ref{L2}:} Let   $\hat{\Lambda}_n $ and $\hat{\mathcal{K}}_n$  be  the empirical  distribution  functions  associated  with  $\Lambda$  and  $\mathcal{K}$  respectively.  Then,  by	the triangle inequality,  we  have   
\begin{equation}
	\Delta\Big[\Lambda \left(w\right) -  \mathcal{K} \left(w\right)\Big] \leq \Delta\Big[\Lambda \left(w\right) - \hat{\Lambda}_n \left(w\right)\Big]  + \Delta\Big[\hat{\mathcal{K}} \left(w\right)-  \mathcal{K} _n\left(w\right)\Big] + \Delta\Big[\hat{\Lambda}_n \left(w\right) - \hat{\mathcal{K}}_n \left(w\right)\Big] \label{eq:2.23}
\end{equation}
Note  by the Glivenko -Canteli Lemma, $\Delta\Big[\Lambda \left(w\right) - \hat{\Lambda}_n \left(w\right)\Big] \to 0$, and $ \Delta\Big[\hat{\mathcal{K}} \left(w\right)-  \mathcal{K} _n\left(w\right)\Big] \to 0$ almost surely,  for  large  $n.$  Therefore,  it    suffices  to  show by, Lemma~\ref{L2}, that for  every  $\epsilon>0$,  there  exists  $N(\epsilon)$  such  that  $n>N(\epsilon),$  implies  

\begin{equation}
	\Delta\Big[\hat{\Lambda}_n \left(w\right) - \hat{\mathcal{K}}_n \left(w\right)\Big]\le \epsilon. \label{eq:2.26}
\end{equation}

Let  $\epsilon>0$   and   choose $N(\epsilon)=[A(m+1/2,m_2,\frak{\nu})/\epsilon]$, then,  for  all  $n\geq N(\epsilon)$   we  have  
\begin{equation}
	\Delta\Big[\hat{\Lambda}_n \left(w\right) - \hat{\mathcal{K}}_n \left(w\right)\Big] \leq \dfrac{A( m_2+1/2, m_2,  \mathfrak{\nu})}{n}\leq \dfrac{A( m_2+1/2, m_2,  \mathfrak{\nu})}{N(\epsilon)}\le  \epsilon,   \label{eq:2.27}
\end{equation}
which  ends  the  proof  of  the Theorem.

{ \bf Proof  of  Theorem~\ref{main1}(ii):} We  observe that  $\mathcal{K}$  is  integral  of  functions  $\phi(w)$   and  therefore  by  the  fundamental  theorem  of  calculus  it  is  absolutely continuous.  Now,  using  \eqref{eq:2.23}  and  \eqref{eq:2.27}  of the  proof  of   Theorem~\ref{main1}(i)  above  we  have  the  desired  results.

\subsection{Proof  of  Theorem ~\ref{cor1} by  \cite[Theorem~7]{Springer1970} }

We  recall  from  Section~\ref{sec1} that $t_i$  is  of  the  form  $\displaystyle t_i = \dfrac{Y_{1(i)}}{Y_{1(i)} +
		Y_2}=W,$  where \(Y_{1(i)} \sim F{(n_1 -\rho+ i,\mathit{\nu})}\)\, and \, \(Y_2 \sim F{(\rho,\mathit{\nu})}\).   Using, Theorem~\ref{main1} we  have  that  each  $t_i$  is  approximately  distributed  as  beta  $B_{t_i}\left(\dfrac{1}{2}(\rho+0.5), \dfrac{1}{2}\rho\right).$  Note  that  the  approximate  distribution  of  $t_i$  is  independent  of  $n_1-\rho+i$  and  therefore p.d.f. of   $\Omega=\prod_{i=1}^{n_2}t_i$,  $\sigma(w)$  is  the Meijer  G-function multiplied  by  a normalizing  constant   $ \Big[
		\frac{\Gamma\left(\frac{1}{2}(2\rho+0.5)\right)}{\Gamma\left(\frac{1}{2}(\rho+0.5)\right)}\Big]^{n_2}.$  Thus,  $g(w)$  is  given  by  $$	\begin{aligned}
			\sigma (\omega)&= \Big[
			\frac{\Gamma\left(\frac{1}{2}(2\rho+0.5)\right)}{\Gamma\left(\frac{1}{2}(\rho+0.5)\right)}\Big]^{n_2}\\
			&\qquad\qquad\qquad\times G^{n_2,0}_{n_2,0}
			\left(\omega \left| \begin{array}{llll}
				\frac{1}{2}(2\rho - 1.5 )
				& \frac{1}{2}(2\rho -1.5 ), & \ldots, & \frac{1}{2}(2\rho -1.5) \\
				\frac{1}{2}(\rho-1.5)& \frac{1}{2}(\rho-1.5), & \ldots, & \frac{1}{2}(\rho-1.5)
			\end{array}  \right. \right),
		\end{aligned}$$
		which  ends  the  proof  of  Theorem~\ref{cor1}(i).	Now  integrating  	$	\sigma (t)$  with  respect  to  $t$,  for  $0<t<w$  and  setting  $n_2=\ell-\rho$  we  have 
	 $$
	 \begin{aligned}
	 &\Sigma(w)=\Big[
	\frac{\Gamma\left(\frac{1}{2}(2\rho+0.5)\right)}{\Gamma\left(\frac{1}{2}(\rho+0.5)\right)}\Big]^{\ell-\rho}\times\\
	&G^{\ell-\rho+1,1}_{\ell-p,1} \left(\omega \left| \begin{array}{llllll}
		1&\frac{1}{2}(2\rho + 0.5 )&\frac{1}{2}(\rho+0.5) , & \ldots, & \frac{1}{2}(2\rho + 0.5)&\frac{1}{2}(2\rho + 0.5) \\
		\frac{1}{2}(\rho+0.5)&	\frac{1}{2}(\rho+0.5)& \frac{1}{2}(\rho+0.5), & \ldots, & \frac{1}{2}(\rho+0.5)& 0
	\end{array}\right.\right)
	\end{aligned}$$
	, which  concludes  the  proof  of  Theorem~\ref{cor1}(ii).

	\section{SIMULATION RESULTS}\label{sec4}
In this section we consider simulation of the required combination of the Fisher-Snedecor distributions to confirm the theories, corollaries and propositions discussed in the previous chapter. Thus for $W = Y_1/\Big(Y_1 + Y_2\Big)$, where $Y_1 \sim F_{{m_1}, \mathfrak{\nu}}$ and $Y_2 \sim F_{{m_2}, \mathfrak{\nu}}$, the variable, $W$ is simulated for some number of times ($n$) and the empirical cumulative distribution functions (\textit{CDF}) estimated for various combinations of $m_1$, $m_2$ and $\mathfrak{\nu}$. 

The distribution $B_w\left(\dfrac{m_2+ 0.5}{2}, \dfrac{m_2}{2}\right)$ is compared thoroughly with actual empirical distribution of the targeted variable $W$.  The comparison are done using graphical approaches (comparing the distribution functions and density functions of the two distributions) and verifications through the Kolmogorov-Smirnov and the Anderson Darling tests of distribution comparison. The maximum distances will also be observed to verify the applicability of the Glivenko-Canteli Lemma in this situation.

\subsection{Comparing the cumulative distribution functions of the Proposed distribution to Proportion of F-variate }

The cumulative distribution of the proposed distribution , 	 $B_w\left(\dfrac{m_2+ 0.5}{2}, \dfrac{m_2}{2}\right)$ and the empirical distribution of the proportion of two F-variates, $W$, are compared on various given values of $m_1$ and $m_2$ and $\mathfrak{\nu}$ and the results in Figures \ref{fig:fig3.1} 

\newpage
\begin{figure}[!htbp]
	\begin{tabular}{p{3.5in}p{3.5in}}
	
		\begin{subfigure}[b]{.35\textwidth}
			\centering
			\hstretch{1.4}{\includegraphics[width=\textwidth]{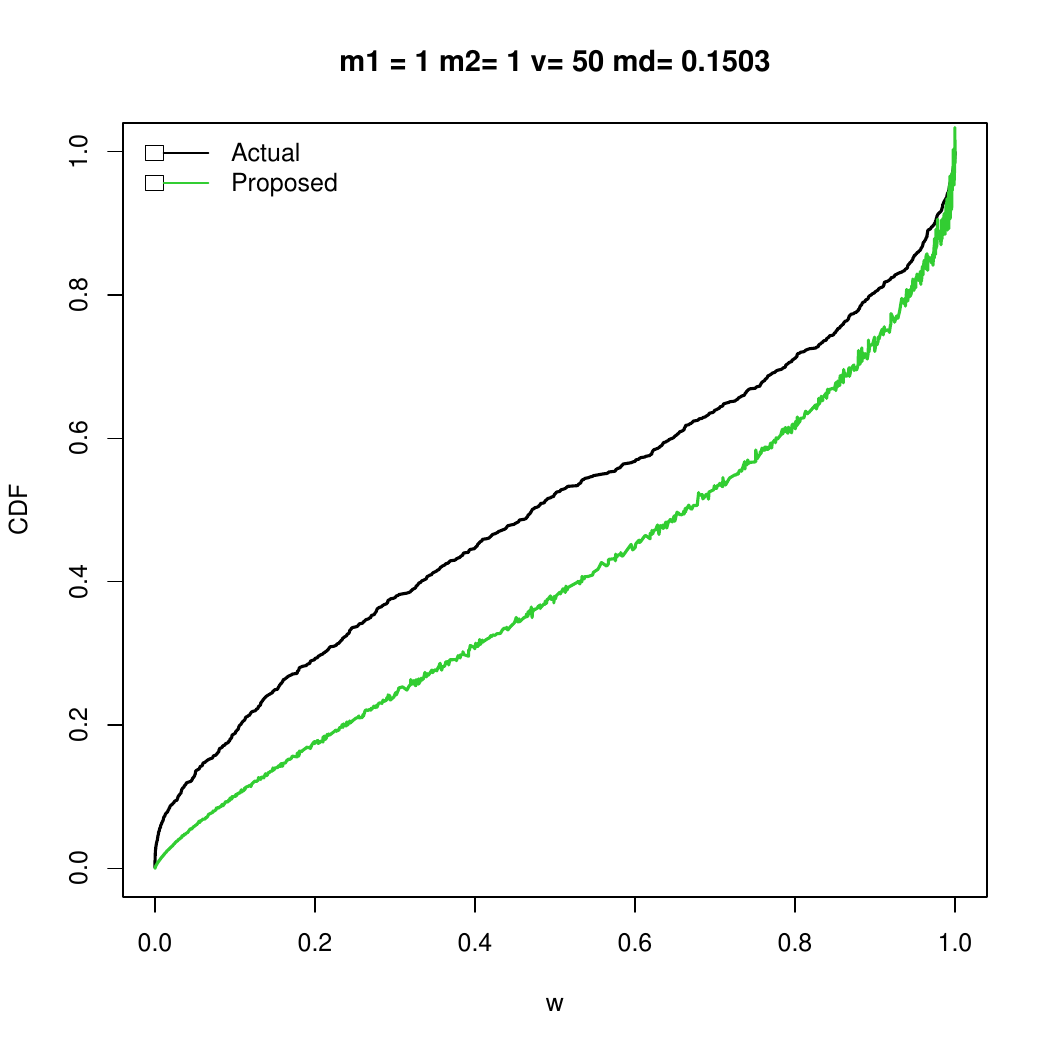}}
			\caption*{$m_1  = 1$, $m_2  = 1$}
		\end{subfigure}&
		\begin{subfigure}[b]{.35\textwidth}
			\centering
			\hstretch{1.4}{\includegraphics[width=\textwidth]{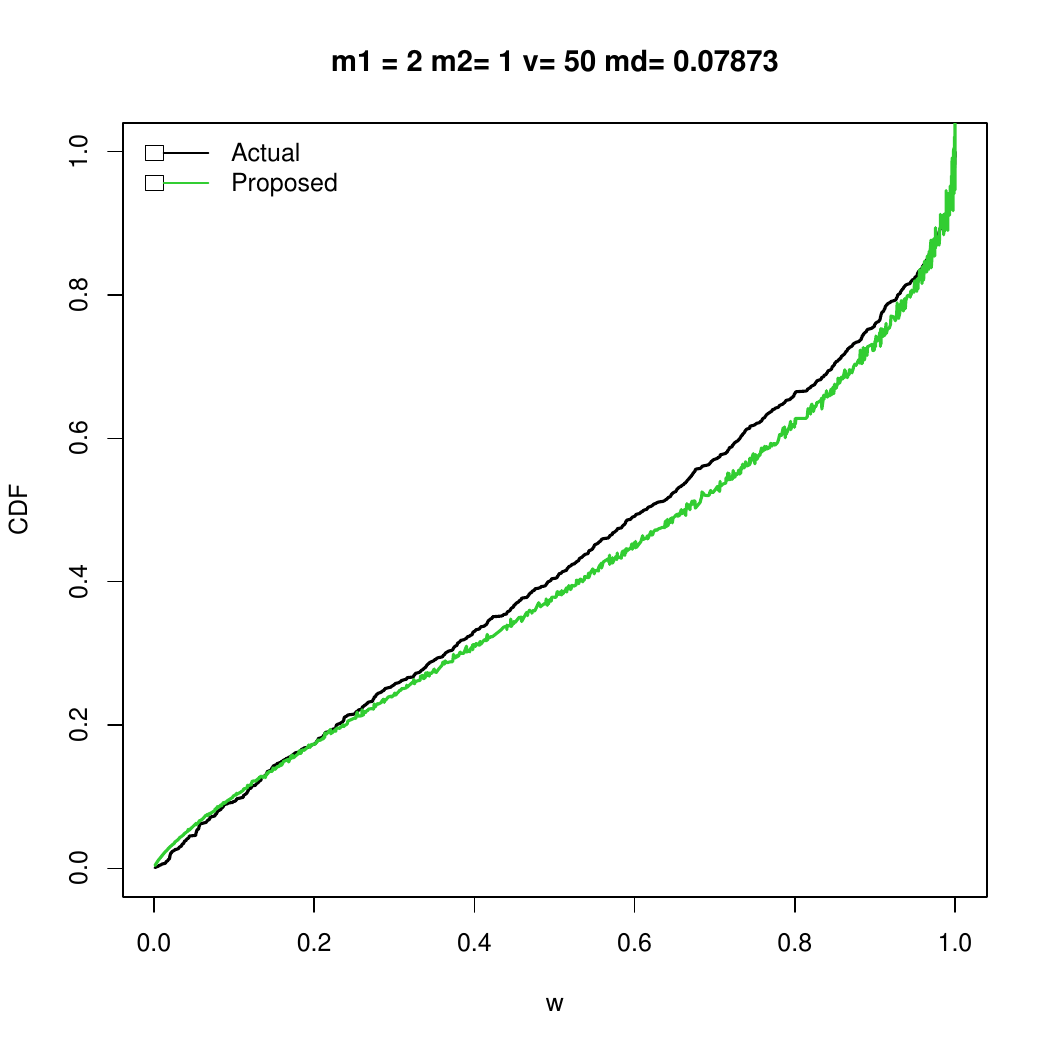}}
			\caption*{$m_1  = 2$, $m_2  = 1$}
		\end{subfigure}\\
			\begin{subfigure}[b]{.35\textwidth}
			\centering
			\hstretch{1.4}{\includegraphics[width=\textwidth]{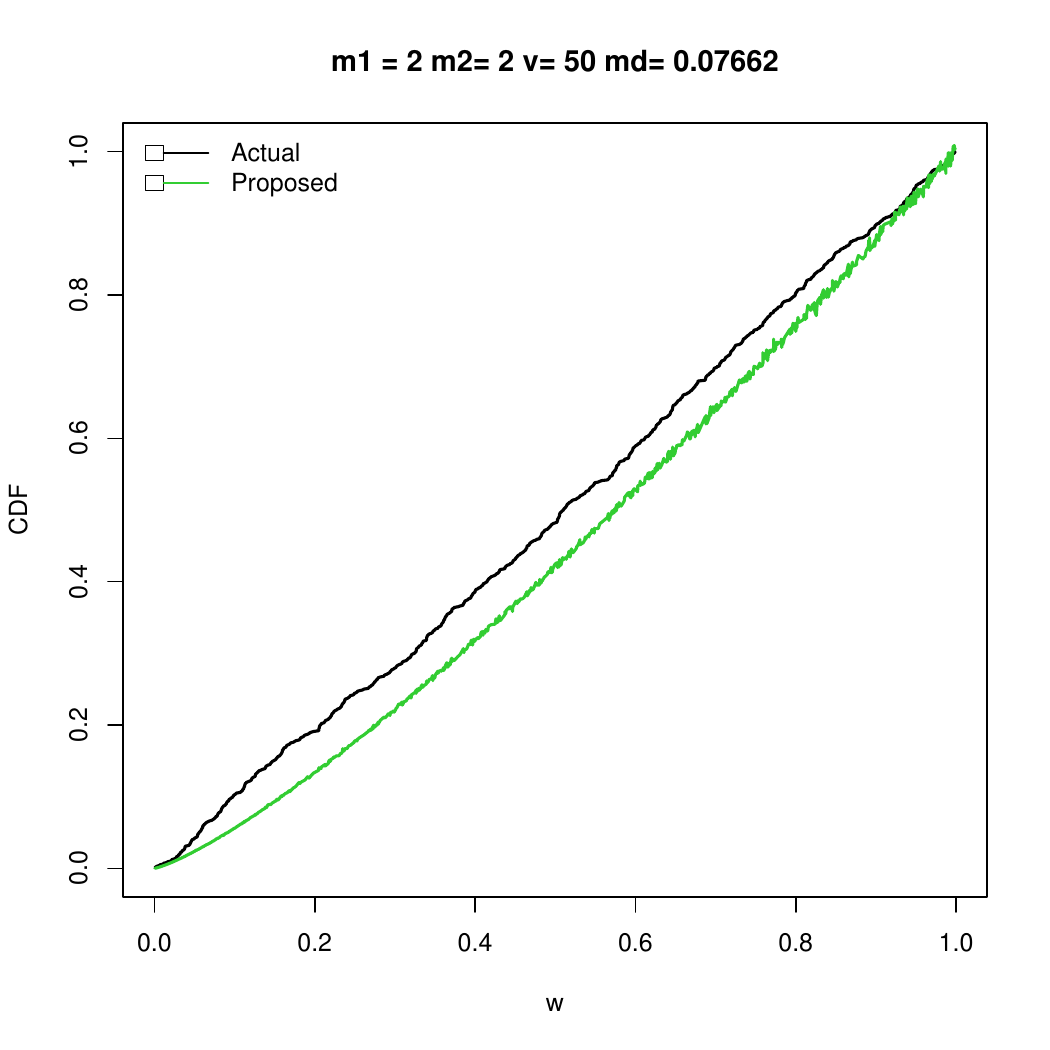}}
			\caption*{$m_1  = 2$, $m_2 = 1$}
		\end{subfigure}&
		\begin{subfigure}[b]{.35\textwidth}
			\centering
			\hstretch{1.4}{\includegraphics[width=\textwidth]{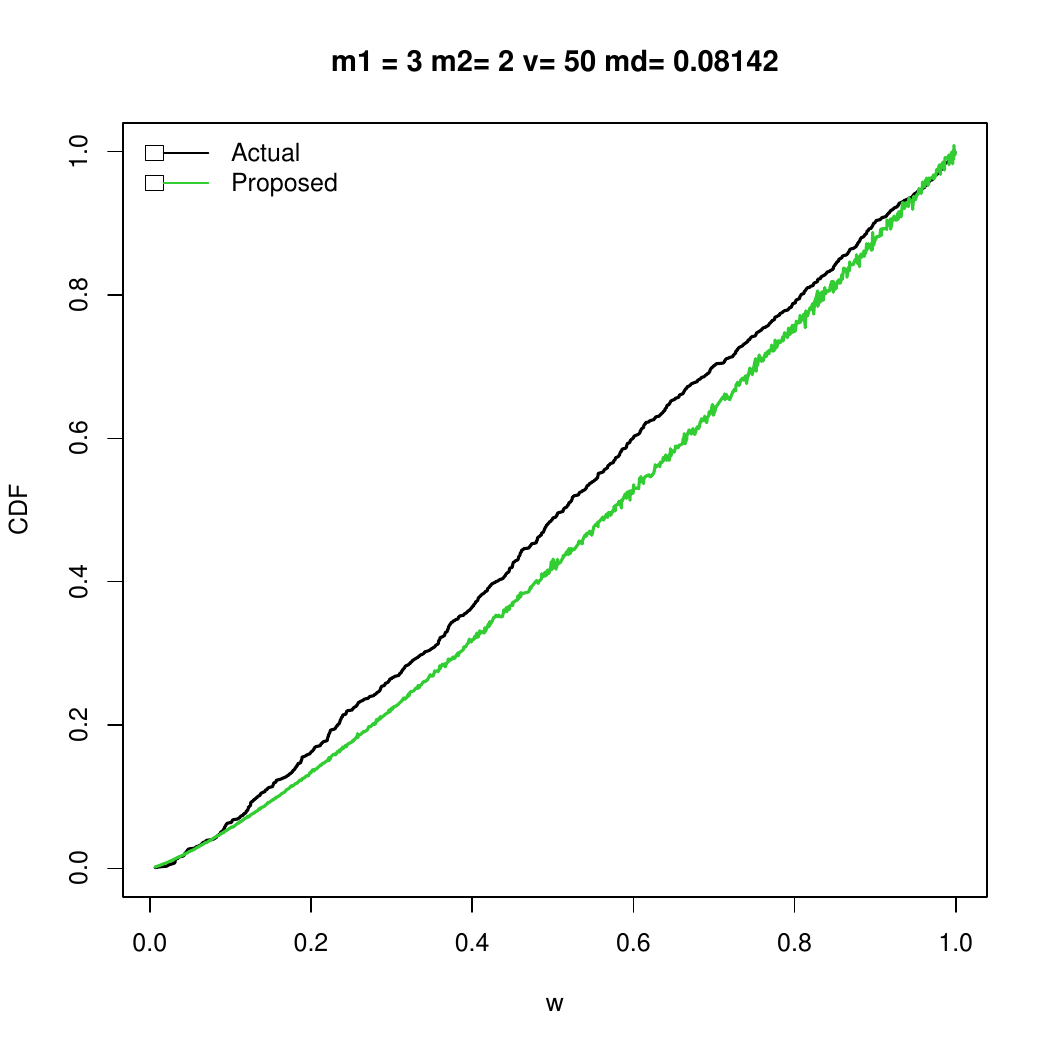}}
			\caption*{$m_1  = 3$, $m_2  = 2$}
		\end{subfigure}\\
	\begin{subfigure}[b]{.35\textwidth}
	\centering
	\hstretch{1.4}{\includegraphics[width=\textwidth]{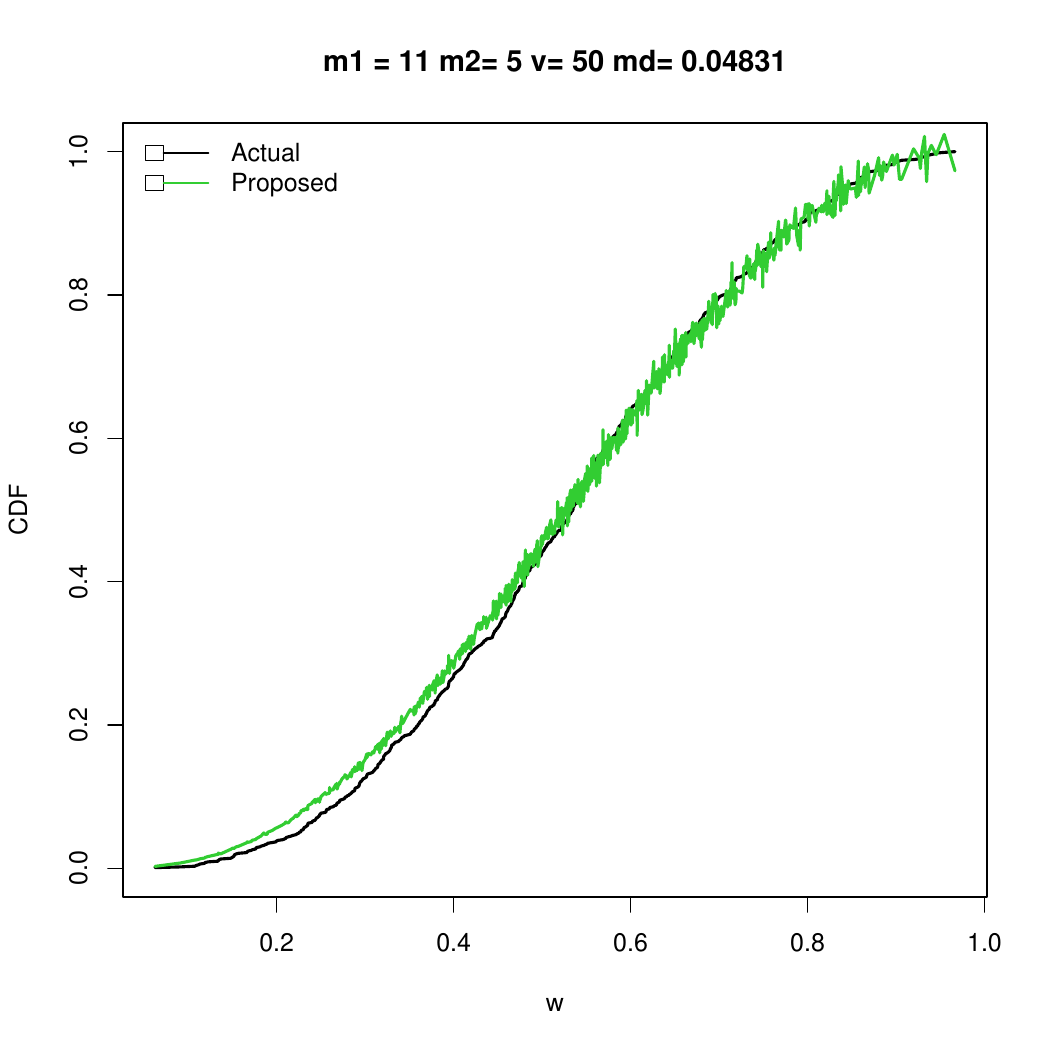}}
	\caption*{$m_1  = 11$ , $m_2  = 5$}
\end{subfigure}&
\begin{subfigure}[b]{.35\textwidth}
	\centering
	\hstretch{1.4}{\includegraphics[width=\textwidth]{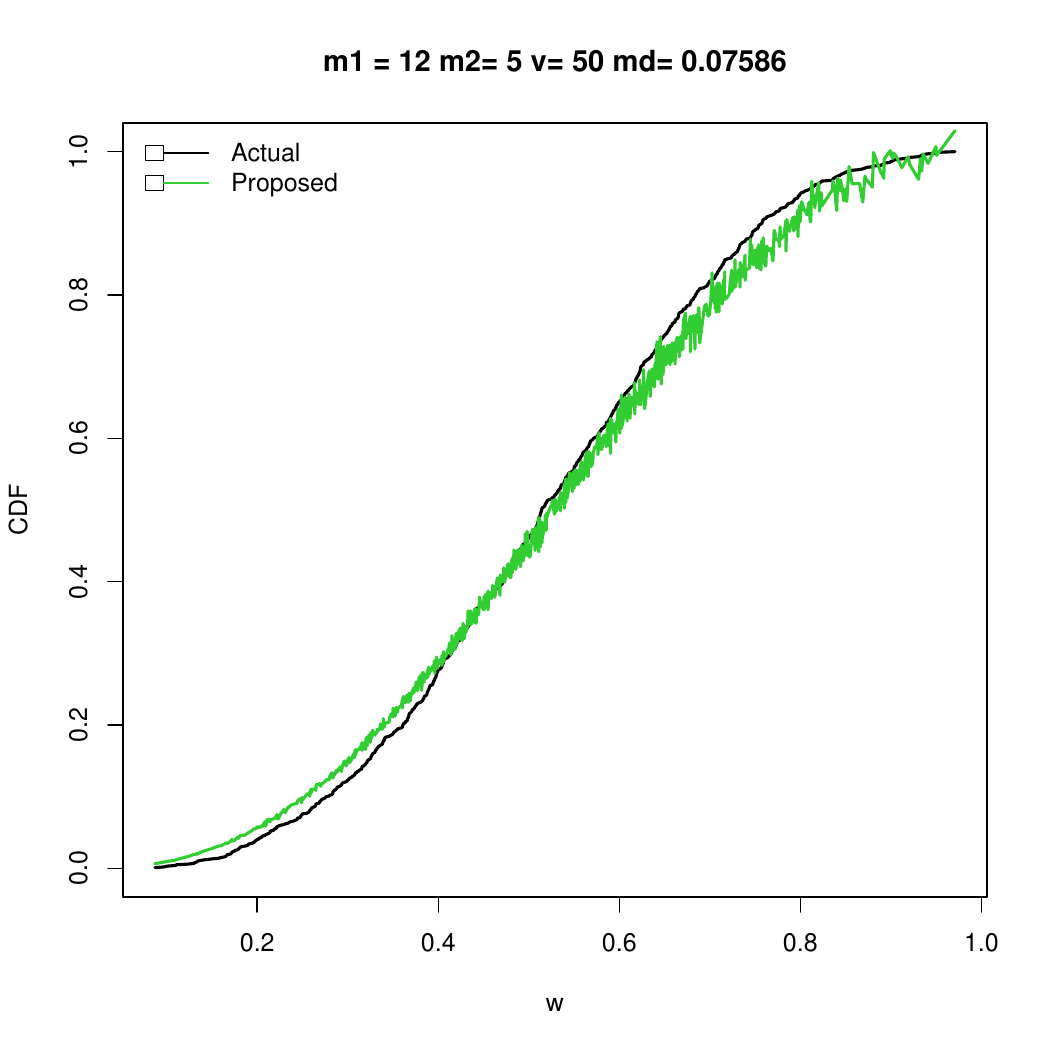}}
	\caption*{$m_1  = 12$, $m_2  = 5$}
\end{subfigure}\\
		\begin{subfigure}[b]{.35\textwidth}
	\centering
	\hstretch{1.4}{\includegraphics[width=\textwidth]{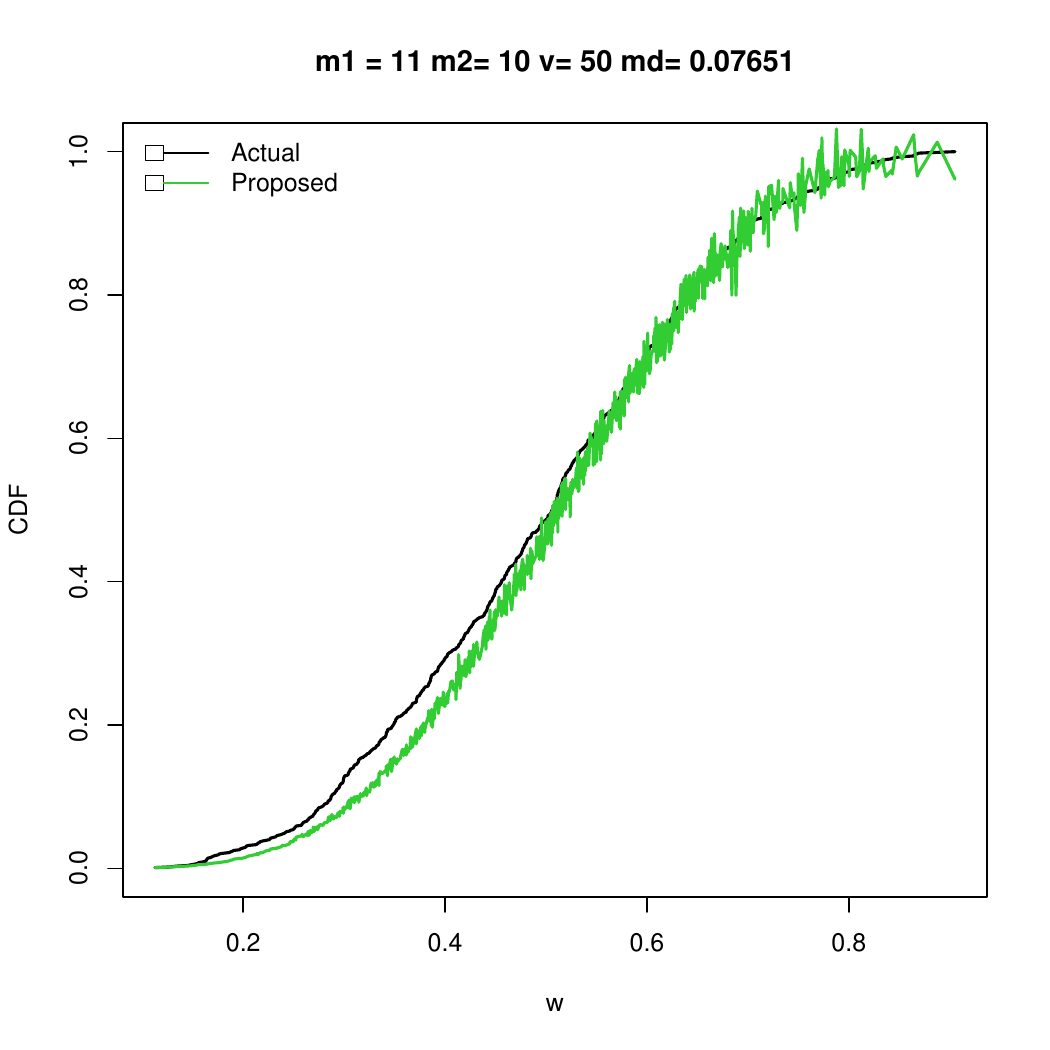}}
	\caption*{$m_1  = 11$, $m_2  = 10$}
\end{subfigure}&
\begin{subfigure}[b]{.35\textwidth}
	\centering
	\hstretch{1.4}{\includegraphics[width=\textwidth]{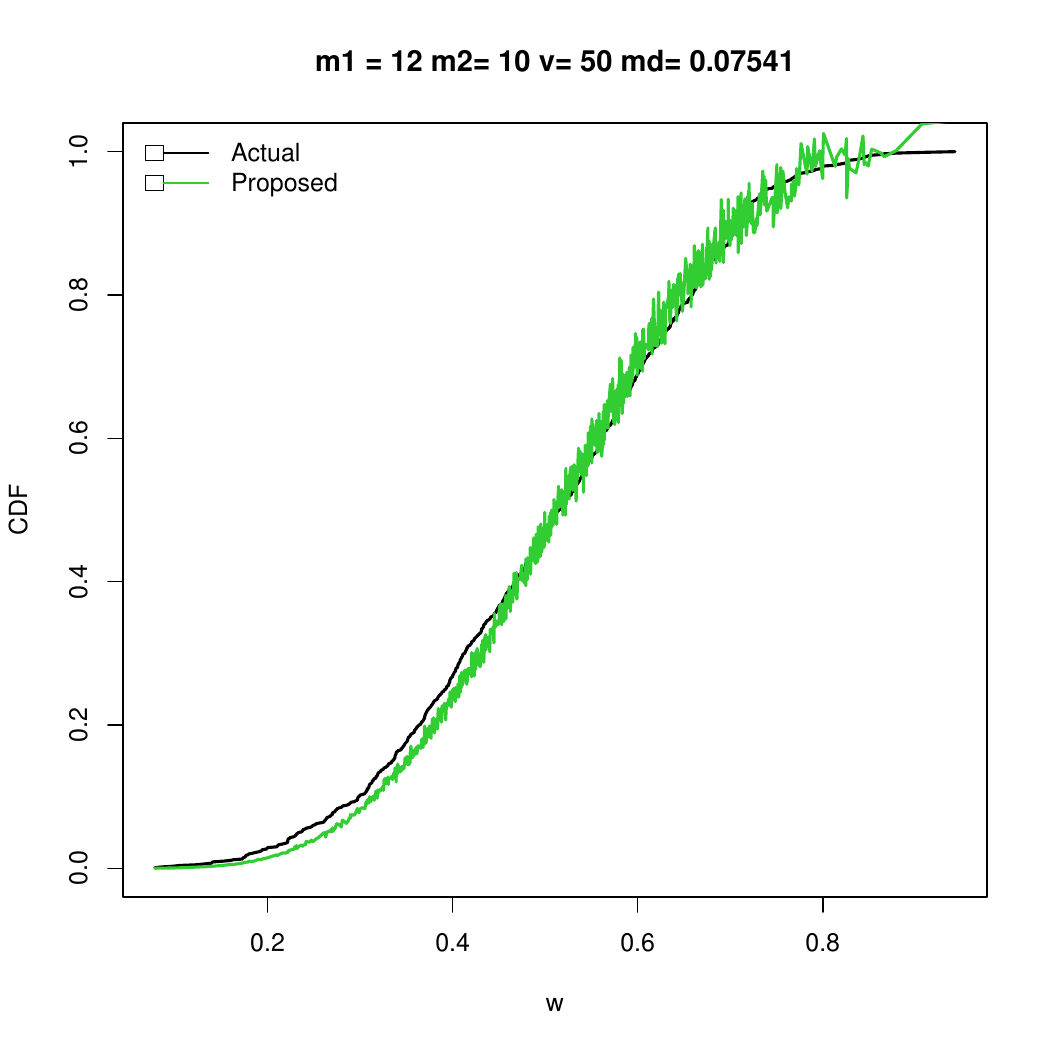}}
	\caption*{$m_1  = 12$, $m_2  = 10$}
\end{subfigure}\\
	\end{tabular}
	\caption{CDF Plots of $W$ and Proposed Distribution for $m_2=1,2 5, 10$, $\mathfrak{\nu} = 50$}
	\centering
	\begin{scriptsize}
		\textit{	md= Maximum Distance Between distribution functions}
	\end{scriptsize}
	\label{fig:fig3.1}
\end{figure}

Figure \ref{fig:fig3.1} directly compare the cumulative function curve of the proposed approximated distribution of $W$ to the empirical distribution curve of $W$ when $m_2 = 1, 2, 5, 10$	 and $\mathfrak{\nu}=50$. The proposed distribution even appears to approximate the actual distribution when $m_1 = m_2$. For all values of $m_1$, the maximum distance is ranged from 0.4 to 0.16. The shapes of the distribution appears to be symmetric too as their distribution curves are almost similar to the forty-five degree line or $\mathcal{S}-$shape. 
The proposed distribution seem to have  approximated the distribution of $W$ better when $m_2$ increases from $1$ to $10$ as the maximum gap reduces as the $m_2$ increased..

\subsection{Distributional comparison tests between Proposed distribution and distribution of proportion of F-variates }
The results Kolmogorov-Smirnov 	and Anderson-Darling tests comparing the two distributions at some defined values of $m_1$, $m_2$ and $\mathfrak{\nu}$ are shown in Table \ref{tab:tab4.1}. The tests were carried out in conditions where the sample size, $n = 200$ .

\begin{table}[!htbp]\small
	\centering
	\caption{Kolmogorov-Smirnov and Anderson-Darling Tests of distributions ($n=200$)}
	\begin{tabular}{rrrrr}
		\Xhline{2pt}
		\multicolumn{1}{l}{m1} & \multicolumn{1}{l}{m2} & \multicolumn{1}{l}{v} & \multicolumn{1}{l}{K-S Test Statistic} & \multicolumn{1}{l}{A-D Test statistic} \\
		\Xhline{2pt}
		3     & 2     & 50    & \textit{\textbf{0.070}} & \textit{\textbf{-0.4497}} \\
		3     & 2     & 150   & \textit{\textbf{0.120}} & \textit{\textbf{2.6462}} \\
		7     & 2     & 50    & \textit{\textbf{0.080}} & \textit{\textbf{-0.0472}} \\
		7     & 2     & 150   & \textit{\textbf{0.100}} & \textit{\textbf{0.0334}} \\
		17    & 2     & 50    & \textit{\textbf{0.105}} & \textit{\textbf{0.7747}} \\
		17    & 2     & 150   & \textit{\textbf{0.100}} & \textit{\textbf{1.1361}} \\
		6     & 5     & 50    & \textit{\textbf{0.165}} & 5.2803 \\
		6     & 5     & 150   & 0.180 & \textit{\textbf{4.0643}} \\
		10    & 5     & 50    & \textit{\textbf{0.110}} & \textit{\textbf{0.7065}} \\
		10    & 5     & 150   & \textit{\textbf{0.115}} & \textit{\textbf{2.1679}} \\
		20    & 5     & 50    & \textit{\textbf{0.095}} & \textit{\textbf{0.7067}} \\
		20    & 5     & 150   & \textit{\textbf{0.085}} & \textit{\textbf{0.4033}} \\
		11    & 10    & 50    & \textit{\textbf{0.085}} & \textit{\textbf{-0.3534}} \\
		11    & 10    & 150   & \textit{\textbf{0.100}} & \textit{\textbf{2.3397}} \\
		15    & 10    & 50    & \textit{\textbf{0.100}} & \textit{\textbf{-0.2185}} \\
		15    & 10    & 150   & \textit{\textbf{0.100}} & \textit{\textbf{0.0764}} \\
		25    & 10    & 50    & \textit{\textbf{0.100}} & \textit{\textbf{0.4349}} \\
		25    & 10    & 150   & \textit{\textbf{0.080}} & \textit{\textbf{-0.6399}} \\
		16    & 15    & 50    & \textit{\textbf{0.130}} & \textit{\textbf{1.8812}} \\
		16    & 15    & 150   & \textit{\textbf{0.110}} & \textit{\textbf{-0.2864}} \\
		20    & 15    & 50    & \textit{\textbf{0.085}} & \textit{\textbf{-0.3123}} \\
		20    & 15    & 150   & \textit{\textbf{0.140}} & \textit{\textbf{2.8565}} \\
		30    & 15    & 50    & \textit{\textbf{0.100}} & \textit{\textbf{-0.2265}} \\
		30    & 15    & 150   & \textit{\textbf{0.075}} & \textit{\textbf{-0.1306}} \\
		26    & 25    & 50    & \textit{\textbf{0.065}} & \textit{\textbf{-0.1677}} \\
		26    & 25    & 150   & \textit{\textbf{0.060}} & \textit{\textbf{-0.8979}} \\
		30    & 25    & 50    & \textit{\textbf{0.120}} & \textit{\textbf{0.4653}} \\
		30    & 25    & 150   & \textit{\textbf{0.120}} & \textit{\textbf{1.5277}} \\
		40    & 25    & 50    & \textit{\textbf{0.090}} & \textit{\textbf{-0.4905}} \\
		40    & 25    & 150   & \textit{\textbf{0.065}} & \textit{\textbf{-0.5917}} \\
		\Xhline{2pt}
	\end{tabular}
	
	\centering
	\footnotesize{\textit{Highlighted test statistics means significant at \(\alpha = 1\%\)}}
	\label{tab:tab4.1}%
\end{table}%

From Table \ref{tab:tab4.1} Kolmogorov-Smirnov test shows  that, the two distributions are identical statistically more than 99\% of the various combinations $m_1, m_2$ and $\mathfrak{\nu}$ at 1\% level of significance . Only one combination of the parameters,  $m_1  =6, m_2 =5$ and $\mathfrak{\nu}=150$. Noting that the K-S test statistic is the supremum of the distance between the two distribution, it can be said that even in the situation where statistical test do not accept the identical distribution the distributions do deviate from each other significantly

The Anderson Darling tests also reports similar results by confirming the identicalness of two distribution in more 99\% of the cases with the only exception being where  $m_1  =6, m_2 =5$ and $\mathfrak{\nu}=50$.

\section{ DISCUSSION  AND CONCLUSION}\label{sec5}
\subsection{Discussion of Results}
Given that $W=\dfrac{Y_1}{Y_1+Y_2}$, $Y_1 \sim F_{{m_1}, \mathfrak{\nu}}$ and $Y_2 \sim F_{{m_2}, \mathfrak{\nu}}$ and $\mathfrak{\nu} > m_1$, the proposed approximate distribution for $W$ is Beta distribution with shape parameters, $\dfrac{m_2+0.5}{2}$ and $\dfrac{m_2}{2}$, denoted by $\Bigg[B_w\left(\dfrac{m_2+0.5}{2}, \dfrac{m_2}{2}\right)\Bigg]$.

The proposed Beta distribution for $W$, $\Bigg[B_w\left(\dfrac{m_2+0.5}{2}, \dfrac{m_2}{2}\right)\Bigg]$,  is found to exhibit same shape in terms of the distribution curves as the empirical distribution of curve of $W$ in the possible combinations of $m_1, m_2$ and $\mathfrak{\nu} =50$ considered. The maximum distance between the distribution function curves is commendably low (mostly less than 0.2) and even works respectably well  for $m_1 \leq m_2$ especially when $m_2 = 1,2,3$.

The distribution, $B_w\left(\dfrac{m_2+0.5}{2}, \dfrac{m_2}{2}\right)$, though is proposed to be just an approximate distribution of $W$, was subjected to distributional tests to compare how statistically significant the identicalness of the proposed distribution to the distribution of $W$. The Kolmogorov-Smirnov and the Anderson-Darling tests were conducted and on different combinations of the parameters. The Kolmogorov-Smirnov tests affirmed the identicalness in more than 90\% of the cases considered and the Anderson Darling more than 90\% of the cases. This was not surprising as the Kolmogorov-Smirnov test is known to be best for distribution with least outliers and the actual distribution of $W$ seem to be platykurtic. Though, the Anderson-Darling tests is best suited for distributions with outliers, it still was able to affirm the identicalness in a substantial number of cases.

The  density curves proposed distribution and the empirical kernel density function curves  of $W$ can also be used for comparison. The trend of the maximum distance between the distribution function of the proposed distribution and the empirical distribution of $W$ can also be used to validate the proposed distribution (loosely based on the Glivenko-Canteli theorem). Interested readers can see these two validation procedures in Thesis, ``Distribution Of Some Diagnostic Statistics in Quantile Regression with t-Distributed Covariate Measurement Errors: A Heuristic Approach" submitted to University of Ghana by Issah Seidu (2023). 
\subsection{Conclusion}
Just as the common scale parameters in proportion of two Gamma variates is completely independent of the shape parameters of resulting Beta distribution, the proportion of two F-variates results in a distribution results in Beta distribution with shape parameters that seem to be functions of only  the numerator degrees of freedom of the second variate. That is, for $W=\dfrac{Y_1}{Y_1+Y_2}$, it is known that  $W\sim B_w \left(m_1, m_2\right)$ if $Y_1 \sim  \Gamma\left(m_1,\mathfrak{\nu}\right)$ and $Y_2\sim  \Gamma\left(m_2,\mathfrak{\nu}\right)$, whereas from this study  $W\simeq B_w\left(\dfrac{m_2+0.5}{2}, \dfrac{m_2}{2}\right)$ if $Y_1 \sim F_{{m_1}, \mathfrak{\nu}}$ and $Y_2 \sim F_{{m_2}, \mathfrak{\nu}}$ provided $m_2\le m_1< \mathfrak{\nu}.$

The Heuristic approach is a very recommendable approach for other functions of variables that seem mathematically not tractable and should be employed more in approximating the distribution of other combinations such as products, ratios etc.

{\bf Declaration}

 The authors report there are no competing interests to declare.

 {\bf Acknowledgement}
 
The  first  author  thanks  the Carnegie  Banga-Africa  Project,  University  of  Ghana  for  supporting  this  Ph.D  research work.

\bibliographystyle{apacite}	
\bibliography{BayesianQR.bib}

\end{document}